\documentclass[11pt]{article}

\usepackage{amssymb}
\usepackage{mathrsfs}
\usepackage{latexsym,amsmath,amssymb,amsfonts,epsfig,graphicx,cite,psfrag}
\usepackage{eepic,color,colordvi,amscd}
\usepackage{ebezier}
\usepackage{verbatim}
\usepackage{subfigure}
\usepackage{accents}
\usepackage{amsthm}
\usepackage{comment}

\theoremstyle{plain}
\newtheorem{theo}{Theorem}[section]
\newtheorem{claim}[theo]{Claim}

\newtheorem{lem}[theo]{Lemma}

\theoremstyle{definition}

\theoremstyle{remark}

\def\qed{\hfill \rule{4pt}{7pt}}

\newcommand{\vertexset}[1]{V(#1)}
\newcommand{\edgeset}[1]{E(#1)}
\newcommand{\order}[1]{|#1|}

\newcommand{\bridges}[2][]{\beta_{#1}(#2)}
\newcommand{\taubridges}[2][]{\tau_{#1 #2}}

\begin{document}

\title{Large cycles in essentially 4-connected planar graphs}
\author{Michael C. Wigal\footnote{Supported by NSF Graduate
    Research Fellowship under Grant No. DGE-1650044}   \ \ and Xingxing Yu\footnote{Partially
    supported by NSF Grant DMS 1600738}\\
School of Mathematics\\
Georgia Institute of Technology\\
Atlanta, GA 30332}
\date{\today}

\maketitle

\begin{abstract}

Tutte proved that every 4-connected planar graph contains a Hamilton cycle, but
there are 3-connected $n$-vertex planar graphs whose longest cycles have length
$\Theta(n^{\log_32})$. On the other hand,  Jackson and Wormald in 1992 proved that an
essentially 4-connected $n$-vertex planar graph contains a cycle of
length at least $(2n+4)/5$, which was recently improved to $5(n+2)/8$ by Fabrici {\it et al}.  In this
paper, we improve this bound to $\lceil (2n+6)/3\rceil$ for $n\ge 6$,
which is best possible,  by proving a quantitative version of a result
of Thomassen on Tutte paths. 

 \bigskip



\end{abstract}

\newpage

\section{Introduction}
The Four Color Theorem \cite{AH77, AHK77}  (also see \cite{RSST97})
states that every plane graph is 4-face-colorable. All known proofs of the Four Color Theorem
require the use of a computer. However, if a plane graph has a
Hamilton cycle then one can  properly  four color all its faces easily.

Tait \cite{Ta84} conjectured that every 3-connected cubic planar 
graph contains a Hamilton cycle, which, if true, would imply the Four
Color Theorem. However, Tutte \cite{Tu46} discovered a counterexample and, since then, families of counterexamples have been constructed, see
for instance \cite{HM88}.
On the other hand, Whitney \cite{Wh31} proved that every planar
triangulation without separating triangles are Hamiltonian, which was
extended by Tutte \cite{Tu56} to all 4-connected planar graphs. Later,
Thomassen \cite{Th83} showed that in fact all 4-connected
planar graphs are Hamilton connected, i.e., there is a Hamilton path between any two vertices.

There has been interest in finding good lower bounds on the circumference
of 3-connected planar graphs. (The {\it circumference} of a graph is the length of a longest cycle
in that graph.) For instance,  Chen and the second author \cite{CY02} showed that the
circumference of a 3-connected planar $n$-vertex graph is at least
$n^{\log_32}$, which is best possible because of iterated planar
triangulations $Tr(k)$: starting with $Tr(0)=K_3$, for each $k\ge 1$, add a 
vertex in each face of $Tr(k-1)$ and connect it with an edge to each vertex on
the boundary of that face.

For any positive integer $k$, a graph is {\it essentially $k$-connected} if it is connected and, for
any $S\subseteq V(G)$ with $|S|<k$, $G-S$ is connected or has exactly
two components one of which has exactly one vertex. Jackson and Wormald
\cite{JW92} proved that the circumference of any  essentially
4-connected $n$-vertex planar graph is
at least $(2n+4)/5$.
Very recently, this bound has been improved to $5(n+2)/8$ by Fabrici,
Harant, Mohr, and Schmidt  \cite{FHMS19}, using sophisticated
discharging rules. The main result of this paper is the
following

\begin{theo}\label{main}
Let $n\ge 6$ be an integer and let
$G$ be any essentially 4-connected  $n$-vertex planar graph. Then
the circumference of $G$ is at least $\lceil (2n+6)/3\rceil $.
\end{theo}

This bound is best possible in the following sense. Take a 4-connected triangulation $T$ on $k$ vertices, and inside each face of $T$ add a
new vertex and three edges from that new vertex to the three vertices in the boundary of that face. The resulting graph, say $G$, has
$n:=3k-4$ vertices. Now take an arbitrary cycle $C$ in $G$. For each $x\in V(C)$ with degree three in $G$,
we delete $x$ from $C$ and add the edge of $G$ between the two neighbors of
$x$ in $C$. This results in a cycle in $T$, say $D$. Clearly,  $\order{D}\le k$;
which implies $\order{C}\le 2k$. Hence, the circumference of $G$ is at most $2k=2(n+4)/3=\lceil (2n+6)/3\rceil$.

Our proof of Theorem~\ref{main} uses an idea from Thomassen
\cite{Th83} and Tutte
\cite{Tu56} --- finding a cycle $C$ in a 2-connected graph $G$ such that every component of $G-C$ has at most
three neighbors on $C$. 
This idea helps us avoid the difficulty of dealing with 4-connected graphs and has
been used before to find Hamilton cycles in 4-connected graphs.
We prove a quantitative version of a result  in  \cite{Th83} on such cycles by also controlling the number of components
of $G-C$, see Theorem~\ref{technical}. Our approach follows that of \cite{Th83}, but many adjustments are needed to complete the work.  

In Section 2, we introduce additional notation and terminology and  state a more technical result,
Theorem~\ref{technical}, from which Theorem~\ref{main} will follow.
In Section 3, we deal with special cases of Theorem~\ref{technical}  when
there exist certain 2-cuts in the graph. In Section 4, we complete the
proof of Theorem~\ref{technical}. In Section 5, we derive Theorem~\ref{main} from
Theorem~\ref{technical}.

We conclude this section with useful notation. We often use $\order{G}$ to denote the number of vertices in $G$, and  
represent a path by a sequence of vertices (with consecutive vertices being adjacent). 
For two graphs $G$ and $H$, we use $G\cup H$ and $G\cap H$ to denote the union and intersection of $G$ and $H$,
respectively.
For any $S\subseteq V(G)$, $G-S$ denotes the subgraph of $G$ obtained from $G$ by deleting all vertices in $S$ and all edges
of $G$ incident with $S$. We often write $G-H$ for $G-V(G\cap H)$. Moreover, for any family $T$ of 2-element subsets of $V(G)$ 
we use $G+T$ to denote the graph with
vertex set $V(G)$ and edge set $E(G)\cup T$. When $T=\{\{u,v\}\}$, we write $G+uv$ instead of  $G+\{\{u,v\}\}$.

For any positive integer $k$ and any graph $G$, a $k$-separation in $G$
is a pair $(G_1,G_2)$ of subgraphs of $G$ such that
$|V(G_1\cap G_2)|=k$, $G=G_1\cup G_2$, $E(G_1)\cap E(G_2)=\emptyset$, and $G_i\not\subseteq G_{3-i}$ for $i=1,2$.
A $k$-cut in $G$ is a set $S\subseteq V(G)$ with $|S|=k$ such that
there exists a separation $(G_1,G_2)$ in $G$
with $V(G_1\cap G_2)=S$ and $G_i-G_{3-i}\ne \emptyset$ for $i=1,2$.

\section{Tutte paths}

In order to find a long cycle in an
essentially 4-connected planar graph, we instead find a cycle $C$ in a
2-connected planar graph $G$ such that every component of
$G-C$ has at most three neighbors on $C$ and the number of
components of $G-C$ is as small as possible. Thus, we introduce the
concept of a
Tutte subgraph.

Let $G$ be a graph and $H\subseteq G$. An {\it $H$-bridge} of $G$ is a
subgraph of $G$ which is either induced by an edge in $E(G)\setminus E(H)$ with both
incident vertices on $H$, or induced by the edges of $G$ that are incident with one
or two vertices in a single component of $G-H$. Let
\begin{itemize}
\item [] $\bridges[G]{H} = |\{B: B \mbox{ is an $H$-bridge of $G$ and }
    |B|\ge 3\}$.  
\end{itemize} 
For any $H$-bridge $B$ of $G$, a vertex in $V(B\cap H)$ is called an {\it attachment} of $B$ on
$H$.   We say that $H$ is a {\it Tutte subgraph} of $G$ if every
$H$-bridge of $G$ has at most three attachments on $H$. Moreover,
for any subgraph $F\subseteq G$, $H$ is said to be an {\it $F$-Tutte}
subgraph of $G$ if $H$ is a Tutte subgraph of $G$ and every $H$-bridge
of $G$ containing an edge of $F$ has at most two attachments on
$H$. (The concept of $F$-Tutte subgraphs is introduced for
induction purpose.) 
A {\it Tutte cycle} (respectively, {\it Tutte path}) is a Tutte
subgraph that is a cycle (respectively, path).

Given a plane graph $G$ and a cycle $C$ in $G$, we say that $(G,C)$ is
a {\it circuit graph} if $G$ is 2-connected, $C$ is the outer cycle of
$G$ (i.e., $C$ bounds the infinite face of $G$), and, for any 2-cut $T$ in $G$, each component of
  $G-T$ must contain a vertex of $C$.  Note that $C$  has a clockwise
orientation and a counterclockwise orientation, and we may use the symmetry between these two orientations.
 For any distinct elements
$x,y\in V(C)\cup E(C)$, we use $xCy$ to denote the subpath of $C$ in
clockwise order from $x$ to $y$ such that $x,y\notin E(xCy)$. We
say that $xCy$ is {\it good} if $G$ has no 2-separation $(G_1,G_2)$
with $V(G_1\cap G_2)=\{s,t\}$ such that $x,s,t,y$ occur on $xCy$ in
order,  $sCt\subseteq G_2$, and $|G_2|\ge 3$. Moreover, let
\[
\taubridges[G]{xy} =
\begin{cases}
2/3, \quad & \text{$xCy$ is not good;}\\
2/3, \quad & \text{$|\{x,y\}\cap E(C)|=1$ and $x$ and $y$ are incident; }\\
1/3, \quad & \text{$|\{x,y\}\cap E(C)|=1$ and $|xCy|=2$;}\\
0, \quad & \text{otherwise.}
\end{cases}
\]

If there is no danger of confusion, we may drop the reference to $G$ in all of the above notation.
We can now state the result from which we will deduce Theorem~\ref{main}.

\begin{theo}\label{technical}
Let $n\ge 3$ be an integer, let $(G,C)$ be a circuit graph on
$n$ vertices, let $u,v \in \vertexset{C}$ be distinct, and
let $e\in  \edgeset{C}$, such that $u,e, v$ occur on $C$ in clockwise order.
Then $G$ has a $C$-Tutte path $P$ between $u$ and $v$ such that $e\in E(P)$ and
$$\bridges{P} \le (n - 6)/3 + \taubridges{vu} + \taubridges{ue} + \taubridges{ev}.$$
\end{theo}

To help the reader digest the notation involved, we illustrate this
statement with two cases: $e=uv$, and $\order{G}=3$.

\begin{lem}\label{base}
Theorem~\ref{technical} holds when $e=uv$ or  $\order{G}=3$.
\end{lem}
\begin{proof}
 As $G$ is 2-connected, we have $|G| \ge 3$.  First, suppose
 $e=uv$. Then $vCu$ is not good because of the 2-separation $(uCv, G-uv)$; so
 $\taubridges{vu}=2/3$. Moreover, since $u,v$ are both incident with
 $e$, $\taubridges{ue}=\taubridges{ev}=2/3$. Hence,   $P:=vu$ gives the
desired path as  $\bridges{P}=1 \le (\order{G} - 6)/3 + \taubridges{vu}
+ \taubridges{ev} + \taubridges{ue}$.

Now assume $e\ne uv$ and $\order{G}=3$. Further assume by symmetry that $u$ is not
incident with $e$. Then  $\taubridges{vu}=0$, $\taubridges{ue}=1/3$, and
$\taubridges{ev}=2/3$. Hence, $P:=C-uv$ gives the desired path as
$\bridges{P}=0=(|G|- 6)/3 + \taubridges{vu}
+ \taubridges{ue} + \taubridges{ev}$.
\end{proof}

\section{Special 2-cuts}

In this section, we deal with two cases when the plane graph $G$ in Theorem~\ref{technical} has certain 2-cuts. In the first case,
$G$ has a 2-cut separating $\{u,v\}$ from $e$. We formulate it as a lemma.

\begin{lem}\label{2cut}
Suppose $n\ge 4$ is an integer and Theorem~\ref{technical} holds for graphs on at most
$n-1$ vertices. Let $(G,C)$ be a circuit graph on $n$ vertices, $u,v \in \vertexset{C}$ be distinct, and
 $e\in  \edgeset{C}$, such that $u,e,v$ occur on $C$ in clockwise order.

 If $G$ has a 2-separation $(G_1,G_2)$ such that $\{u,v\}\subseteq V(G_1)$, $\{u,v\}\not\subseteq V(G_2)$,  $e\in E(G_2)$,
 and $\order{G_2}\ge 3$,  then
$G$ has a $C$-Tutte path $P$ between $u$ and $v$ such that $e\in E(P)$ and
$\bridges{P} \le (n - 6)/3 + \taubridges{vu} + \taubridges{ue} + \taubridges{ev}$.
 \end{lem}

\begin{proof}
Let $\vertexset{G_1 \cap G_2} = \{x,y\}$ with $x \in V(eCv)$ and $y \in V(uCe)$.  
See Figure~\ref{2cut1}. Let $G'_i = G_i + xy$ for $i \in \{1,2\}$ such that $G_1'$ is a
plane graph with outer cycle $C_1:=xCy+yx$ and $G_2'$ is a plane graph with
outer cycle $C_2:=yCx+xy$.  Note that both $(G_1',C_1)$ and $(G_2',C_2)$ are circuit graphs.  Let
$e_1:=xy$,  $n_1 := |G_1'|$, and $n_2: = |G_2'|$. Then  $n_1 + n_2 = n
+ 2$. Since $\{u,v\}\not\subseteq V(G_2)$, we may assume by symmetry that $u\ne y$.

\begin{figure}[ht]
\begin{center}
        \includegraphics[scale=0.3]{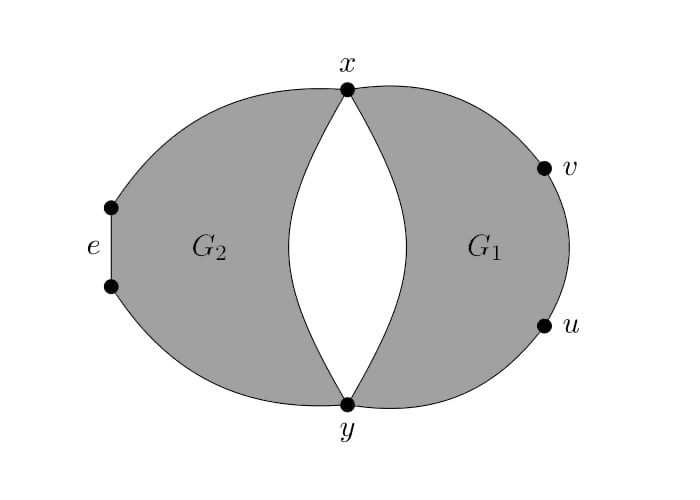}
           \caption{The separation $(G_1,G_2)$ in $G$. }
             \label{2cut1}
\end{center}
    \end{figure}

By assumption,  $G_1'$ has a $C_1$-Tutte path between $u$ and $v$ such that $e_1\in
E(P_1)$ and  $$ \bridges[G_1']{P_1} \le (n_1 - 6)/3 +
\taubridges[G_1']{vu} + \taubridges[G_1']{ue_1} +
\taubridges[G_1']{e_1v},$$
and  $G_2'$ has a $C_2$-Tutte path $P_2$ between $x$ and $y$ such that $e\in E(P_2)$
and  $$\bridges[G_2']{P_2} \le (n_2 - 6)/3 + \taubridges[G_2']{xy} + \taubridges[G_2']{ye} + \taubridges[G_2']{ex}.$$

Note that  $P := (P_1 \cup P_2) - e_1$ is a $C$-Tutte
path in $G$ between $u$ and $v$ such that $e\in E(P)$. Moreover, $\taubridges[G_1']{vu} = \taubridges[G]{vu}$ and
$\taubridges[G_2']{xy} = 0$. Thus,  
  \[
         \bridges[G]{P}  = \bridges[G_1']{P_1} + \bridges[G_2']{P_2}\le (n-6)/3 -4/3+\taubridges[G]{vu} + 
\taubridges[G_1']{ue_1} +\taubridges[G_1']{e_1v} +\taubridges[G_2']{ye}+ \taubridges[G_2']{ex}.
  \]       

We claim that $\taubridges[G_1']{e_1v}+\taubridges[G_2']{ex}\le \taubridges[G]{ev}+2/3$. This is clear if $ \taubridges[G]{ev}=2/3$. 
 If $\taubridges[G]{ev}=1/3$ then $|eCv|=2$ and, hence, $|eCx|=2$ or
 $|e_1Cv|=2$;  so $\taubridges[G_1']{e_1v}=1/3$ or
 $\taubridges[G_2']{ex}=1/3$, and the inequality holds as well. 
 Now assume $\taubridges[G]{ev}=0$. Then $|eCv|\ge 3$ and $eCv$ is
 good in $G$. So $|e_1Cv|\ge 3$ and $e_1Cv$ is good in $G_1'$, or 
$|eCx|\ge 3$ and $eCx$ is good in $G_2'$, or $|e_1Cv|=|eCx|=2$. Hence,
$\taubridges[G_1']{e_1v}=0$, or $\taubridges[G_2']{ex}=0$, or $\taubridges[G_1']{e_1v}=\taubridges[G_2']{ex}=1/3$. 
Again we see that  the inequality holds.

Similarly,  $\taubridges[G_1']{ue_1}+\taubridges[G_2']{ye}\le \taubridges[G]{ue}+2/3$.  So 
$\bridges[G]{P}\le  (n-6)/3 +\taubridges[G]{vu} + \taubridges[G]{ue} + \taubridges[G]{ev}$.
\end{proof}

The next lemma deals with a different type of 2-cuts in the graph $G$ in Theorem~\ref{technical}.

\begin{lem}\label{2cut-1}
Suppose $n\ge 4$ is an integer and Theorem~\ref{technical} holds for graphs
on at most
$n-1$ vertices. Let $(G,C)$ be a circuit graph on
$n$ vertices,   $u,v \in \vertexset{C}$ be distinct, and
 $e=xy\in  \edgeset{C}$, such that $u,x,y,v$ occur on $C$ in clockwise order.

If $\{u,x\}$ or $\{v,y\}$ is a 2-cut in $G$ then $G$ has a $C$-Tutte
path $P$ between $u$ and $v$ such that $e\in E(P)$ and $\bridges[G]{P}\le
(n-6)/3+\taubridges[G]{vu}+\taubridges[G]{ue}+\taubridges[G]{ev}.$
\end{lem}

\begin{proof}   Suppose $\{u,x\}$ or $\{v,y\}$ is a 2-cut in $G$, say
  $\{u,x\}$ by symmetry. See Figure~\ref{2cut2}. Then $G$ has a
2-separation $(G_1,G_2)$ such that $xCu\subseteq  G_1$, $uCx\subseteq G_2$,
and $\order{G_2}\ge 3$. We choose $(G_1,G_2)$ so that $G_2$ is maximal. Then $ux\notin E(G_1)$. Note that $\taubridges[G]{ue}=2/3$.

\begin{figure}[ht]
\begin{center}
        \includegraphics[scale=0.3]{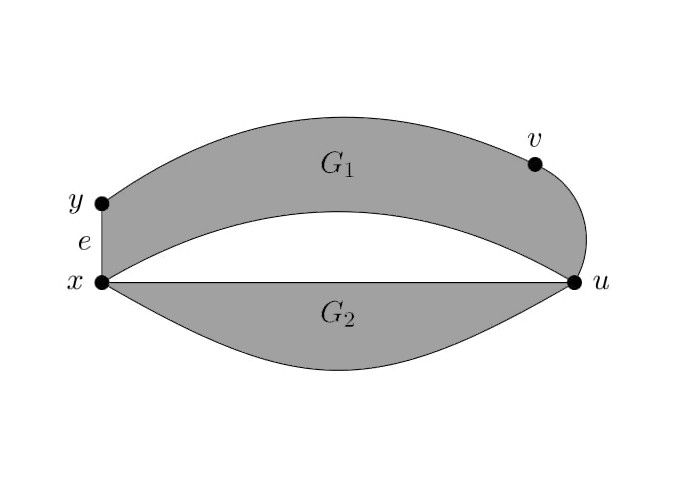}
           \caption{The separation $(G_1,G_2)$ in $G$. }
             \label{2cut2}
\end{center}
    \end{figure}

\medskip

{\it Case} 1.  $G_1$ is 2-connected.

Then let $C_1$ denote the outer cycle of $G_1$. Since $(G,C)$ is a
circuit graph, $(G_1,C_1)$ is a circuit graph.
By assumption,
$G_1$ has a $C_1$-Tutte path $P_1$ between $u$ and $v$ such that $e\in E(P_1)$ and
\[
 \bridges[G_1]{P_1}\le (\order{G_1}-6)/3+\taubridges[G_1]{vu}+\taubridges[G_1]{ue}+\taubridges[G_1]{ev}.
\]
 Note that
$\taubridges[G_1]{vu}=\taubridges[G]{vu}$, $\taubridges[G_1]{ue}=0$ (as $ux\notin E(G_1)$), 
and $\taubridges[G_1]{ev}=\taubridges[G]{ev}$. So 
\[
\bridges[G]{P_1}=\bridges[G_1]{P_1}+1\le (\order{G}-6)/3+\taubridges[G]{vu}+\taubridges[G]{ue}+\taubridges[G]{ev},
\]
and $P_1$ gives the desired path $P$.

\medskip

{\it Case} 2. $G_1$ is not 2-connected.

Let $G_1':=G_1+ux$ be the plane graph with outer cycle $C_1:=xCu+ux$,
and let $G_2':=G_2+xu$ be the plane graph with outer cycle $C_2:=uCx+xu$.
Since $(G,C)$ is a circuit graph, we see that both $(G_1',C_1)$ and $(G_2',C_2)$
are circuit graphs.

Note that $\taubridges[G_1']{vu}=\taubridges[G]{vu}$,
$\taubridges[G_1']{ue}=1/3$,  and $\taubridges[G_1']{ev}=
\taubridges[G]{ev}$.
By assumption, $G_1'$ has a $C_1$-Tutte path $P_1$ between $u$ and
$v$ such that $e\in E(P_1)$ and
\[
 \bridges[G_1']{P_1}\le
 (\order{G_1'}-6)/3+\taubridges[G_1']{vu}+\taubridges[G_1']{ue}+\taubridges[G_1']{ev}
=
(\order{G_1}-6)/3+\taubridges[G]{vu}+(\taubridges[G]{ue}-1/3)+\taubridges[G]{ev}.
\]
Since $G_1$ is not 2-connected,
$ux\in E(P_1)$. 

Choose $e'\in E(uC_2x)$ such that
$\taubridges[G_2']{e'x}=1/3$ and  $\taubridges[G_2']{ue'}\le
2/3$.  Note that $\taubridges[G_2']{xu}=0$. By assumption, $G_2'$ has a $C_2$-Tutte path
$P_2$ between $x$ and $u$ such that $e'\in E(P_2)$ and
\[
\bridges[G_2']{P_2} \le
(\order{G_2'}-6)/3+\taubridges[G_2']{xu}+\taubridges[G_2']{ue'}+\taubridges[G_2']{e'x}
\le (\order{G_2}-6)/3+1.
\]
Now $P:=(P_1-ux)\cup P_2$ is a $C$-Tutte path in $G$ between $u$
and $v$ such that $e\in E(P)$. Moreover, 
\begin{align*}
 \bridges[G]{P} & =\bridges[G_1']{P_1}+\bridges[G_2']{P_2}\\
   & \le (\order{G_1}-6)/3+\taubridges[G]{vu}+(\taubridges[G]{ue}-1/3)+\taubridges[G]{ev}+(\order{G_2}-6)/3+1\\
   &
   <(n-6)/3+\taubridges[G]{vu}+\taubridges[G]{ue}+\taubridges[G]{ev}.
\end{align*}
So $P$ is the desired path.
\end{proof}

\section{Proof of Theorem~\ref{technical}}
We apply induction on $n$. 
By Lemma~\ref{base} and by symmetry, we may assume that $u$ is not incident with $e$,
$|G|=n\ge 4$, and  the assertion
holds for graphs on at most $n-1$ vertices. Let $e=v'v''$ such that
$u,v',v'',v$ occur on $C$ in clockwise order.
Then by Lemma~\ref{2cut-1},
\begin{itemize}
\item [(1)] neither $\{u,v'\}$ nor $\{v,v''\}$ is a 2-cut in $G$.
\end{itemize}

Moreover, by Lemma~\ref{2cut}, we may assume that $G$ has no 2-cut $T$
such that $T\ne \{u,v\}$ and $T$
separates $e$ from $\{u,v\}$. Thus, by planarity, $uCe$ is contained
in a block of $G-eCv$, which is denoted by $H$.
See Figure~\ref{structure}.  Note that $H\cong K_2$ or $H$ is 2-connected. We may assume that

\begin{itemize}
\item [(2)] $H$ is 2-connected.
\end{itemize}
For, suppose that $H\cong K_2$.
Then $v'$ must have degree 2 in $G$ and $G-v'$ is 2-connected; for otherwise, by planarity,  there
exists a vertex $z\in V(v''Cv)$ such that $\{v',z\}$ is a 2-cut in $G$
separating $e$ from $\{u,v\}$, contradicting  Lemma~\ref{2cut}.
Let $C':=v''Cu+uv''$ be the outer cycle of the plane graph $G':=(G-v')+uv''$, and let
$e':=uv''$. Note that $(G',C')$ is a circuit graph,  $\taubridges[G']{ue'}=2/3=\taubridges{ue}+1/3$, $\taubridges[G']{e'v}=\taubridges[G]{ev}$, and
$\taubridges[G']{vu}=\taubridges[G]{vu}$. Hence, by induction hypothesis, $G'$ has a $C'$-Tutte path $P'$ between $u$ and $v$ such that
$e'\in E(P')$ and
$$\bridges[G']{P'} \le (\order{G'} - 6)/3 +\taubridges[G']{vu}+  \taubridges[G']{ue'}+\taubridges[G']{e'v} =(n-6)/3
+ \taubridges[G]{vu} + \taubridges[G]{ue} + \taubridges[G]{ev}.$$
Now $P:=(P'-e')\cup uv'v''$ is a $C$-Tutte path in $G$ between $u$ and $v$ such that $e\in E(P)$ and
$\bridges[G]{P} \le (n-6)/3
+ \taubridges[G]{vu} + \taubridges[G]{ue} + \taubridges[G]{ev}$. $\Box$

\begin{figure}[ht]
\begin{center}
        \includegraphics[scale=0.3]{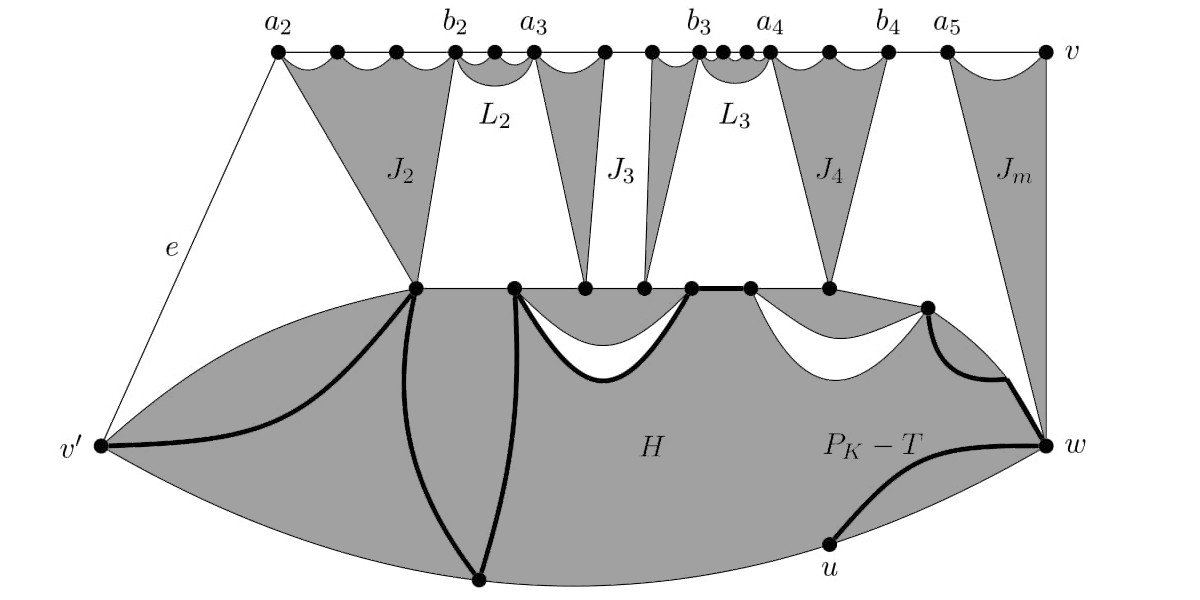}
           \caption{The subgraph $H$ of $G$ and the bridges between $eCv$ and $H$. }
             \label{structure}
\end{center}
    \end{figure}

By (2), let $C'$ denote the outer cycle of $H$. Our strategy is
to use induction hypothesis to find a path in $H$ and extend it to the desired
path in $G$ along $eCv$. 
To do so, we need to avoid double counting too many vertices and, hence,
we will need to contract some subgraphs of $H$.
A 2-separation $(H_1,H_2)$ in $H$ with $V(H_1\cap H_2)\subseteq V(v'C'u)$ is said to be
{\it maximal} if there is no 2-separation $(H_1',H_2')$ in $H$ with
$V(H_1'\cap H_2')\subseteq V(v'C'u)$, such that
the subpath of $v'C'u$ between the two vertices in $V(H_1'\cap H_2')$ properly contains the
subpath of $v'C'u$ between the two vertices in $V(H_1\cap H_2)$.

\begin{figure}[ht]
\begin{center}
        \includegraphics[scale=0.3]{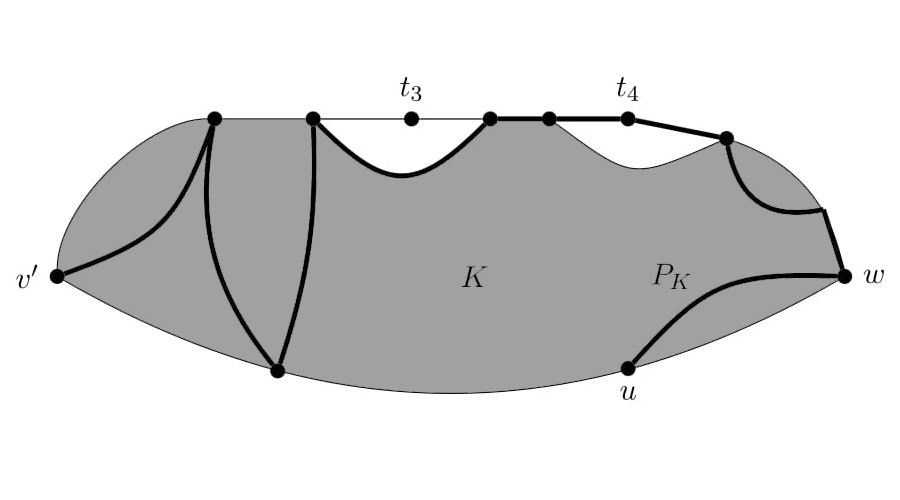}
           \caption{The graph $K$.}
             \label{graphK}
\end{center} 
    \end{figure}

Let $K$ be obtained from $H$ as follows:  For every maximal
2-separation $(H_1,H_2)$ in $H$ with $H_1$ containing $uCv'$,
contract $H_2$ to a single vertex (i.e., replace $H_2$ by a path of length 2
between the vertices of $V(H_1\cap H_2)$).  See Figure~\ref{graphK}.   Let $T$ denote the set of
the new vertices resulted from such contractions. Note that each vertex of $T$ has degree
2 in $K$. Let $D$ be the outer cycle of $K$. Then  $uDv'=uCv'$.  Let $w\in V(vCu)$ such that $wC'v'=wCv'$. Let 
$w'=w$ if $w\in V(K)$; and otherwise let $w'\in T$ be the vertex resulted from the contraction of such an $H_2$ containing $w$.
We may assume that

\begin{itemize}
\item [(3)] $K$ contains a $D$-Tutte path $P_K$ between $u$ and $v'$
   such that $w'\in V(P_K)$ and 
\[
\bridges[K]{P_K} =
  \begin{cases}
     (\order{K} -  6)/3 + \taubridges[G]{ue}+1, \quad & \text{if $|T|\ge 2$,}\\
     (\order{K} -  6)/3 + \taubridges[G]{ue}+ 2/3, \quad & \text{if $|T|\le 1$.}
  \end{cases}
\]
\end{itemize}
First, suppose $\taubridges[K]{uv'}<\taubridges[G]{ue}$. Note that
$|v'Du|\ge 3$ by (1), (2), and planarity. So we may choose $e'\in E(v'Du)$ with the following property:
$|v'De'|\ge 2$; if $|T|\ge 2$ then $e'$
is incident with $w'$; and if $|T|=1$ then  $e'$ is incident with the vertex in $T$ and  $e'$ is incident with $w'$ (whenever possible).  

Then $\taubridges[K]{v'e'}\le 1/3$ when $|T|\le 1$.   By induction hypothesis,
$K$ contains a $D$-Tutte path $P_K$ between $u$ and $v'$ such that $e'\in E(P_K)$ and
\[
\bridges[K]{P_K} \le (\order{K} - 6)/3
+\taubridges[K]{uv'}+\taubridges[K]{v'e'} + \taubridges[K]{e'u}\le
  \begin{cases}
    (\order{K} - 6)/3+\taubridges[G]{ue}+1, \quad & \text{if $|T|\ge 2$,}\\
    (\order{K} - 6)/3+\taubridges[G]{ue}+2/3, \quad & \text{if $|T|\le 1$.}
  \end{cases}
\]
By the choice of $e'$, we see that $w'\in V(P_K)$; so (3) holds. 

Thus, we may assume that $\taubridges[K]{uv'}\ge
\taubridges[G]{ue}$. Then $\taubridges[K]{uv'}=0$ (so $\taubridges[G]{ue}=0$), or $\taubridges[K]{uv'}=2/3$
and $uCv'$ is not good (so $\taubridges[G]{ue}=2/3$). Hence, 
$\taubridges[K]{uv'}= \taubridges[G]{ue}$  and  $|uCe|\ge 3$.

Suppose  $T\ne \emptyset$ and let $t\in T$ such that $t\in N_K(w')\cup \{w'\}$ whenever possible. Let $N_K(t)=\{x,y\}$ with
 $v', x,t,y,u$ occurring  on $D$ in clockwise order.
Let $K':=(K-t)+xy$ and $D':=yDx+xy$, such that $K'$ is a plane graph
and $D$ is its outer cycle. Then $(K',D')$ is a circuit graph.  Let $e':=xy$. Note that
$\taubridges[K']{uv'} =\taubridges[K]{uv'}= \taubridges[G]{ue}$.
By induction hypothesis, $K'$ contains a $D'$-Tutte path $P'$ between $u$ and $v'$ such that $e'\in E(P')$ and
\[
\bridges[K']{P'} \le (\order{K'} - 6)/3
+\taubridges[K']{uv'}+\taubridges[K']{v'e'} + \taubridges[K']{e'u}=
(\order{K} - 6)/3 -1/3+\taubridges[G]{ue}+\taubridges[K']{v'e'}+\taubridges[K']{e'u}.
\]
In particular,  $\bridges[K']{P'} \le(\order{K} - 6)/3+\taubridges[G]{ue}+1$. 
Note that $w'\in V(P')$ by the choice of $t$. Let  $P_K:=(P'-e')\cup xty$. Now, if $|T|\ge 2$
then $P_K$ is the desired
path for (3).
Hence we may assume $|T|=1$. If $u\ne y$ then $\taubridges[K']{e'u}\le 1/3$ (as $|T|=1$);  so
  $\bridges[K']{P'} \le(\order{K} -  6)/3+\taubridges[G]{ue}+2/3$ and $P_K$ is the desired path.
Now assume $u=y$. Then by (1),  $x\ne v'$ and, hence,
$\taubridges[K']{v'e'}\le 1/3$ (as $|T|=1$). So  $\bridges[K']{P'} \le (\order{K} -
  6)/3+\taubridges[G]{ue}+2/3$;  again $P_K$ is the desired path for (3).

Hence, we may assume $T=\emptyset$. 
If $\order{v'Du}\ge 4$ then we may choose $e'\in E(v'Du)$ such that
$\taubridges[K]{v'e'}=0$ and $\taubridges[K]{e'u}\le 2/3$; so by induction
hypothesis, $K$ has a $D$-Tutte path $P_K$ between $u$ and $v'$ such that $e'\in E(P_K)$ and $\bridges[K]{P_K}\le
(\order{K}-6)/3+\taubridges[K]{uv'}+2/3=
(\order{K}-6)/3+\taubridges[G]{ue}+2/3$, and (3) holds as $w'\in
V(P_K)$ (since $T=\emptyset$).  Thus, we may assume
$\order{v'Du}=3$ and let $x\in V(v'Du)\setminus \{u,v'\}$.
 Then $w'\in \{u,x\}$.

Since $|uCv'|\ge 3$ (as $\taubridges[G]{ue}\ne 1/3$),  we may choose $f\in E(uCv')$ such
that  $\taubridges[K]{fv'}=1/3$ and $\taubridges[K]{uf}\le 2/3$. Note that
$\taubridges[K]{v'u}=0$ (as $T=\emptyset$). So by induction hypothesis, $K$ has a $D$-Tutte path $P_K$
between $u$ and $v'$ such that $f\in E(P_K)$ and $\bridges[K]{P_K}\le (\order{K}-6)/3+1$. Note $w'\in V(P_K)$ as $T=\emptyset$. Thus, we may assume $\taubridges[G]{ue}=0$, for,
otherwise, $\bridges[K]{P_K}\le (\order{K}-6)/3+\taubridges[G]{ue}+2/3$, and (3) holds. 
So $uCe$ is good in $G$.

Since $T=\emptyset$,
$x$ has a neighbor in $K-v'Du$. Let $K':=(K-v'x)+uv'$ be the plane
graph whose outer cycle $D'$ consists of $uv'$ and the path in the outer walk of $K-v'x$ from $v'$ to $u$
and containing $x$.
Then $(K',D')$ is a circuit graph, since $uCe$ is good in $G$.
 Let $e':=xu$. Then $\taubridges[K']{uv'}=0$,
$\taubridges[K']{v'e'}=0$, and $\taubridges[K']{e'u}=2/3$.
By induction hypothesis, $K'$ has a
$D'$-Tutte path $P_K$ between $u$ and $v'$ such that $e'\in E(P_K)$ and
 \[
\bridges[K']{P_K} \le (\order{K'} - 6)/3+\taubridges[K']{uv'}+\taubridges[K']{v'e'} + \taubridges[K']{e'u}=
(\order{K} - 6)/3 +\taubridges[G]{ue}+2/3.
\]
Clearly, $P_K$ is also a $D$-Tutte path in $K$ and
$\bridges[K]{P_K}=\bridges[K']{P_K}$. Since $w'\in \{u,x\}$, $w'\in V(P_K)$; so $P_K$ is the desired path for (3). $\Box$

\medskip

We wish to extend $P_K$ along $eCv$ to the desired path $P$ in $G$. Thus we
need a useful description of the structure of the part of $G$ 
that lies between $H$ and $eCv$. See Figure~\ref{structure} for an illustration. 

Let ${\cal B}$ be the set of $(H \cup eCv)$-bridges of $G$. Then
$G = H \cup eCv \cup (\bigcup_{B\in {\cal B}}B)$. Since $H$ is a block
of $G-eCv$, $\order{B\cap
H}\le 1$ for all $B\in {\cal B}$. Note that each vertex $t\in T$ corresponds to a $(P_K-T)$-bridge of $H$ whose attachments on 
$P_K-T$ are the neighbors of $t$ in $P_K$, and that all other $(P_K-T)$-bridges of $H$ are also $P_K$-bridges of $K$.  

 For  $B_1,B_2\in {\cal B}$ with $|B_1 \cap H| = |B_2 \cap H|=1$, we write $B_1 \sim B_2$ if $V(B_1\cap
 H)=V(B_2\cap H)\subseteq V(P_K- T)$, or
 if there exists a $(P_K-T)$-bridge $B$ of $H$ such that $V(B_1\cap H)\cup
 V(B_2\cap H)\subseteq V(B-P_K)$.
Clearly, $\sim$ is an equivalence relation on ${\cal B}$.
Let ${\cal B}_i$, $i=1, \ldots, m$, be the equivalence classes of
${\cal B}$ with respect to $\sim$, such that $H\cap \left(\bigcup_{B\in {\cal B}_i}B\right)$ 
occur on $D$ from $v'$ to $w$  in order $i=1, \ldots, m$, with $v'\in
V(B)$ for all $B\in {\cal B}_1$ and $w\in V(B')$ for some $B'\in {\cal B}_m$.
Let $a_i,b_i\in V(eCv)$ such that
\begin{itemize}
\item [(a)] $a_i\in V(B)$ for some $B\in {\cal
  B}_i$ and $b_i\in V(B')$ for some $B'\in {\cal B}_i$ (possibly $B=B'$),
\item [(b)]  $v'', a_i,b_i,v$  occur on $eCv$ in order, and
\item [(c)] subject to (a) and (b),  $a_iCb_i$ is maximal.
\end{itemize}
Note that $v''=a_1$ and $v=b_m$.  Let $J_i$ denote the union of $a_iCb_i$, all
members of ${\cal B}_i$, those $(H\cup eCv)$-bridges of $G$ whose
attachments are all contained in $a_iCb_i$, and, if applicable, also the $(P_K-T)$-bridge of
$H$ containing $B\cap H$ for all $B\in {\cal B}_i$. Note that  $|J_i\cap (P_K-T)|\in \{1,2\}$, and if  $|J_i\cap
(P_K-T)|=2$ we let $t_i\in T$ be the vertex corresponding to  the $(P_K-T)$-bridge of
$H$ contained in $J_i$.  
For $1 \le i < m$,  let $L_i$ denote the union of
$b_iCa_{i+1}$ and those $(H\cup eCv)$-bridges of $G$ whose
attachments are all contained in $b_iCa_{i+1}$. Let ${\cal L}=\{L_i: 1\le i<m\}$.

 By Lemma~\ref{2cut}, we have $\order{J_1}=2$.  Thus, letting
 $P_1=J_1$, we have

\begin{itemize}
\item [(4)]  $\bridges[J_1]{P_1}=0=(|J_1|-1)/3 -1/3$.
\end{itemize}

For $1 < i < m$, let
\begin{itemize}
   \item ${\cal J}_1 =\{J_i: |J_i\cap (P_K-T)|=1  \mbox{ and }  a_i\ne b_i\}$,
   \item ${\cal J}_2 = \{J_i: |J_i\cap (P_K-T)|=2  \mbox{ and } t_i\notin
     V(P_K)\}$, and
\item ${\cal J}_3 =\{J_i: |J_i\cap (P_K-T)|=2 \mbox{ and } t_i\in
  V(P_K)\}$.
  
\end{itemize}

\begin{itemize}
\item [(5)] For $J_i \in {\cal J}_1$,
$J_i$ has a path $P_i$ between  $a_i$ and $b_i$  such that
  $P_i\cup (J_i\cap P_K)$ is an $a_iCb_i$-Tutte subgraph of $J_i$ and
\[
     \bridges[J_i]{P_i\cup (J_i\cap P_K)}\le
\begin{cases}
 (\order{J_i}-2)/3-1/3, \quad & \text{if $eCv$ is good,}\\
(\order{J_i}-2)/3, \quad & \text{otherwise.}
\end{cases}
\]
\end{itemize}
Let $V(J_i\cap P_K)=\{x\}$.
Consider the plane graph $J_i':=J_i+a_ix$ whose outer cycle $C_i$
consists of $a_iCb_i$,  the edge $e_i:=xa_i$,
and the path in the outer walk of $J_i$ between $b_i$ and $x$ not
containing $a_i$.
Then $(J_i', C_i)$ is a circuit graph. Note that
$\taubridges[J_i']{xe_i}=2/3$
and $\taubridges[J_i']{b_ix}=0$.

Hence, by induction hypothesis, $J_i'$ has a $C_i$-Tutte path $P_i'$ between $x$ and $b_i$ such that
$e_i\in E(P_i')$ and
$\bridges[J_i']{P_i'}\le
(\order{J_i}-6)/3+\taubridges[J_i']{e_ib_i}+2/3$.
Note that $\taubridges[J_i']{e_ib_i}\le 2/3$ and if $eCv$ is good in $G$ then
$\taubridges[J_i']{e_ib_i}\le 1/3$ (as $a_i\ne b_i$).
So $P_i:=P_i'-x$ gives the desired path for (5). $\Box$

\begin{itemize}
\item [(6)] For $J_i \in {\cal J}_2$, $J_i$ has a path $P_i$ between
  $a_i$ and $b_i$  such that
  $P_i\cup (J_i\cap (P_K-T))$ is an $a_iCb_i$-Tutte subgraph of $J_i'$ and
\[
\bridges[J_i]{P_i\cup (J_i\cap (P_K-T))}\le
  \begin{cases}
  (\order{J_i}-4)/3+1/3, \quad & \text{if $eCv$
                           is good,}\\
  (\order{J_i}-4)/3+1, \quad & \text{otherwise.}
  \end{cases}
\]
\end{itemize}
Let $V(J_i\cap (P_K-T))=\{x,y\}$ such that $v',y,x,w$ occur on $D$ in clockwise order. Let $J_i'$ be the block
of $J_i-\{x,y\}$ containing $a_iCb_i$, and let $C_i$ be the outer
cycle of $J_i'$. Note that $a_iC_ib_i=a_iCb_i$.

By planarity there exists a vertex $z\in V(b_iC_ia_i)\setminus \{a_i,b_i\}$ such that $b_iC_iz-z$
contains no neighbor of $y$ and $zC_ia_i-z$ contains no
neighbor of $x$.

First, suppose  $b_iC_ia_i$ is good in $J_i'$.  Let $e_i\in E(b_iC_ia_i)$ incident with
$z$ such that $\taubridges[J_i']{e_ia_i}\le 1/3$ or
$\taubridges[J_i']{b_ie_i}\le 1/3$.
Now by induction hypothesis, $J_i'$ has a $C_i$-Tutte path $P_i$ between
$a_i$ and $b_i$ such that $e_i\in E(P_i)$ and  $\bridges[J_i']{P_i}\le
(|J_i'|-6)/3+ \taubridges[J_i']{a_ib_i}+1$.
If $J_i'\ne J_i-\{x,y\}$ then $|J_i'|\le |J_i|-3$ and, since
$b_iC_ia_i$ is good in $J_i'$, 
\begin{align*}
  \bridges[J_i]{P_i\cup (J_i\cap (P_K-T))}
      & =\bridges[J_i']{P_i}+1 \\
      & \le   (|J_i|-3-6)/3+\taubridges[J_i']{a_ib_i}+1+1\\
     &=(|J_i|-4)/3+\taubridges[J_i']{a_ib_i}+1/3.
\end{align*}
If $J_i'=J_i-\{x,y\}$ then $|J_i'|=|J_i|-2$ and, since $b_iC_ia_i$ is
good in $J_i'$, 
\begin{align*}  \bridges[J_i]{P_i\cup (J_i\cap (P_K-T))} & = \bridges[J_i']{P_i}\\
           &\le (|J_i|-2-6)/3+\taubridges[J_i']{a_ib_i}+1\\
          & =(|J_i|-4)/3+\taubridges[J_i']{a_ib_i}-1/3.
\end{align*}
 Since $\taubridges[J_i']{a_ib_i}=0$ (if $eCv$ is good)  and
$\taubridges[J_i']{a_ib_i}\le 2/3$ (if $eCv$ is not good), we see that
$P_i$ gives the desired path for (6).

Now assume that $b_iC_ia_i$ is not good in $J_i'$. Then let
$(M_1,M_2)$ be a 2-separation in $J_i'$ such that $a_iCb_i\subseteq
M_1$, $|M_2|\ge 3$, $z\in M_2$ (whenever possible), and,  subject to these conditions, $M_2$ is
minimal. Let $V(M_1\cap
M_2)=\{z_1,z_2\}$ such that $a_i,b_i,z_1,z_2$ occur on $C_i$ in
clockwise order. Let  $M_1':=M_1+z_1z_2$ be the plane graph with outer cycle  
$D_1:=z_2C_iz_1+z_1z_2$, and let $M_2':=M_2+z_2z_1$ be the plane graph with outer cycle $D_2:=z_1C_iz_2+z_2z_1$. Then
$(M_1',D_1)$ and $(M_2',D_2)$ are circuit graphs. Let $f:=z_1z_2$.

By induction hypothesis, $M_1'$ has a $D_1$-Tutte path $R_1$ between
$a_i$ and $b_i$ such that $f\in E(R_1)$ and 
$$\bridges[M_1']{R_1}\le (|M_1'|-6)/3+ \taubridges[M_1']{a_ib_i}+4/3=   (|M_1'|-6)/3+ \taubridges[J_i']{a_ib_i}+4/3 .$$ 
Also by induction
hypothesis and choosing an edge $f'\in E(z_1C_iz_2)$ so that
$\taubridges[M_2']{z_1f'}\le 1/3$ or $\taubridges[M_2']{f'z_2}\le 1/3$,
we see that  $M_2'$ has a $D_2$-Tutte path $R_2$ between $z_1$ and $z_2$
such that $f'\in E(R_2)$ and $$\bridges[M_2']{R_2}\le
(|M_2'|-6)/3+1,$$ as $\taubridges[M_2']{z_2z_1}=0$ (since
$(M_2',D_2)$ is a circuit graph).   Let $P_i=(R_1-f)\cup R_2$, which is a path
in $J_i'$ between $a_i$ and $b_i$ such that $P_i\cup (J_i\cap
(P_K-T))$ is an $a_iCb_i$-Tutte subgraph of $J_i'$. Note that  $z\in V(P_i)$ by the
choice of $(M_1,M_2)$ (that $z_2\in V(M_2)$ whenever possible and $M_2$ is minimal).

If $J_i'=J_i-\{x,y\}$ then $|J_i'|=|J_i|-2$ and 
\begin{align*}
   \bridges[J_i]{P_i\cup (J_i\cap (P_K-T))} & =\bridges[M_1']{R_1}+ \bridges[M_2']{R_2} \\
                               & \le  (|M_1'|-6)/3+
                               \taubridges[J_i']{a_ib_i}+4/3+
                               (|M_2'|-6)/3+1\\
                               & =(|J_i'|-6)/3+\taubridges[J_i']{a_ib_i}+1\\
                               &=
                               (|J_i|-4)/3+\taubridges[J_i']{a_ib_i}-1/3.
\end{align*}
If $J_i'\ne J_i-\{x,y\}$ then $|J_i'|\le |J_i|-3$ and 
\begin{align*}
   \bridges[J_i]{P_i\cup (J_i\cap (P_K-T))} & =\bridges[M_1']{R_1}+ \bridges[M_2']{R_2} +1\\
                               & \le  (|M_1'|-6)/3+
                               \taubridges[J_i']{a_ib_i}+4/3+
                               (|M_2'|-6)/3+2\\
                               & =(|J_i'|-6)/3+\taubridges[J_i']{a_ib_i}+2\\
                               &\le
                               (|J_i|-4)/3+\taubridges[J_i']{a_ib_i}+1/3.
\end{align*}

Therefore, since $\taubridges[J_i']{a_ib_i}=0$ if $eCv$ is good and
$\taubridges[J_i']{a_ib_i}\le 2/3$ otherwise, we see that 
$P_i$ is the desired path for (6). $\Box$

\begin{itemize}
\item [(7)] For $J_i \in {\cal J}_3$, 
$J_i$ has disjoint paths $P_i,P_i'$   such that $P_i$ is between $a_i$
and $b_i$, $P_i'$ is between the two vertices in $V(J_i\cap (P_K-T))$,
  $P_i\cup P_i'$ is an $a_iCb_i$-Tutte subgraph of $J_i$, and
\[
   \bridges[J_i]{P_i\cup P_i'}\le
 \begin{cases}
     (\order{J_i}-4)/3-1/3,  \quad & \text{if $eCv$ is good,}\\
     (\order{J_i}-4)/3,  \quad & \text{otherwise.}
  \end{cases}
\]
\end{itemize}
Let $V(J_i\cap (P_K-T))=\{x,y\}$ and assume that $v',y,x,w$ occur on $D$
in clockwise order. Consider the plane graph $J_i':=J_i+b_ix$  with $a_iCb_i,
e_i:=b_ix, y$ occur on its outer cycle $C_i$ in clockwise order. (Note
that $xy\notin E(J_i)$ by the definition of $J_i$.) Then, since $(G,C)$ is a circuit
graph, $(J_i',C_i)$ is a circuit graph and  $\taubridges[J_i']{e_iy}=\taubridges[J_i']{ya_i}=0$.

Thus, by induction hypothesis, $J_i'$ contains a $C_i$-Tutte path
$R_i$ between $a_i$ and $y$ such that $e_i\in E(R_i)$ and
\[
 \bridges[J_i']{R_i}\le (|J_i'|-6)/3+\taubridges[J_i']{a_ie_i}=(|J_i|-4)/3+\taubridges[J_i']{a_ie_i}
-2/3.
\]
Thus, $R_i-e_i$ is an $a_iCb_i$-Tutte subgraph of $J_i$ such that 
\[
  \bridges[J_i]{R_i-e_i}=\bridges[J_i']{R_i}\le
  (|J_i|-4)/3+\taubridges[J_i']{a_ie_i} -2/3.
\]
Note that $\taubridges[J_i']{a_ie_i}\le 1/3$ (if $eCv$ is good) and
$\taubridges[J_i']{a_ie_i}\le 2/3$ (if $eCv$ is not good). 
 So $R_i-e_i$ gives the desired paths for (7). $\Box$

\medskip

Next, we consider $J_m$. Note that if $|J_m\cap (P_K-T)|=2$ then $t_m=w'\in V(P_K)$, and if $|J_m\cap (P_K-T)|=1$ then $w\in V(J_m\cap (P_K-T))$. 
Let $c=2$ if  $|J_m\cap (P_K-T)|=1$, and  $c=4$ if  $|J_m\cap (P_K-T)|=2$. Note that $c$ is the number of vertices double counted by $|J_m|$ 
 and $|K\cup (J_1\cup L_1) \cup \ldots \cup (J_{m-1}\cup L_{m-1})|$.

\begin{itemize}
\item [(8)] $J_m$ has disjoint paths $P_m,P_m'$ with  $P_m$ between $a_m$  and $b_m$ and $P_m'$ between the vertices in 
$V(J_m\cap (P_K-T))$, such that that  $P_m\cup P_m'$ is an $(a_mCu\cap J_m)$-Tutte subgraph of $J_m$ and  
  \[
      \bridges[J_m]{P_m\cup P_m'}\le
\begin{cases}
      (|J_m|-c)/3+\taubridges[G]{vu} -1/3, \quad & \text{if $a_m\ne
        b_m$ and $eCv$ is good,}\\
      (|J_m|-c)/3+\taubridges[G]{vu}, \quad & \text{otherwise.}
\end{cases}
\]
\end{itemize}
First, suppose  $a_m=b_m$ and $|J_m\cap (P_K-T)|=1$. Then $c=2$, and let $P_m=a_m$ and $P_m'=w$. Now
$\bridges[J_m]{P_m\cup P_m'}\le 1$, with equality only if
$|J_m|\ge 3$, in which case, $vCu$ is not good in $G$ and
$\taubridges[G]{vu}=2/3$. So  $\bridges[J_m]{P_m\cup P_m'}\le
(|J_m|-c)/3+\taubridges[G]{vu}$ as $c=2$.

Now assume $a_m\ne b_m$ or  $|J_m\cap (P_K-T)|=2$. If  $|J_m\cap (P_K-T)|=2$ then let $V(J_m\cap (P_K-T))=\{x,y\}$ such that $v,y,w,x$ occur on $D$ 
in clockwise order, and if  $|J_m\cap (P_K-T)|=1$ then let $y=x=w$.
Consider the plane graph $J_m^*:=J_m+ya_m$ with outer
cycle $C_m$ containing $a_mCx$ and $ya_m$. Then $(J_m^*,C_m)$ is a
circuit graph.  Let $e_m:=ya_m$. 
Note that $\taubridges[J_m^*]{b_mx}\le \taubridges[G]{vu}$,  
$\taubridges[J_m^*]{ye_m}=2/3=(4-c)/3$, and if $x\ne y$ then $\taubridges[J_m^*]{xy}=0=(4-c)/3$.

By induction hypothesis, $J_m^*$ has a $C_m$-Tutte
path $P_m^*$  between $b_m$ and $x$  such that $e_m\in E(P_m^*)$ and
\begin{align*}
\bridges[J_m^*]{P_m^*} &\le
(\order{J_m^*}-6)/3+\taubridges[G]{vu}+(4-c)/3+\taubridges[J_m^*]{e_mb_m}\\
& = (\order{J_m}-c)/3-2/3 +\taubridges[G]{vu}+\taubridges[J_m^*]{e_mb_m}.
\end{align*}
Note that $\taubridges[J_m^*]{e_mb_m} \le 1/3$ (when $eCv$ is good in $G$) and
$\taubridges[J_m^*]{e_mb_m}\le 2/3$ (when $eCv$ is not good in $G$). Hence,
$P_m^*-ya_m$ gives the desired paths for (8). $\Box$

\medskip

Next, we consider the family ${\cal L}:=\{ L_i: 1\le i<m\}$, see its
definition in front of (4).

\begin{itemize}
\item [(9)] For each $L_i\in {\cal L}$, $L_i$ contains a $b_iCa_{i+1}$-Tutte path $Q_i$ from $b_i$
  to $a_{i+1}$ such that $\bridges[L_i]{Q_i}\le \max\{0, (|L_i|-2)/3-1/3\}$.
\end{itemize}
If $|b_iCa_{i+1}|\le 2$ then let $Q_i:=b_iCa_{i+1}$; we see that
$\bridges[L_i]{Q_i}=0$ as $(G,C)$ is a circuit graph.
So assume $|b_iCa_{i+1}|\ge 3$. Then consider the plane graph
$L_i':=L_i+a_{i+1}b_i$ with outer cycle $D_i:=b_iCa_{i+1}+a_{i+1}b_i$. Note
that $(L_i',D_i)$ is a circuit graph. Choose an edge
$e_i\in E(b_iCa_{i+1})$ so that
$\taubridges[L_i']{b_ie_i}=1/3$. Note that
  $\taubridges[L_i']{a_{i+1}b_i}=0$ and
  $\taubridges[L_i']{e_ia_{i+1}}\le 2/3$.  So
by induction hypothesis, $L_i'$ contains a $D_i$-Tutte path $Q_i$ between $b_i$ and $a_{i+1}$ such that
$e_i\in E(Q_i)$ and
$\bridges[L_i']{Q_i}\le
(|L_i'|-6)/3+1 =(|L_i|-2)/3-1/3$. $\Box$

\medskip

Let $P$ be the union of $P_K-T$, $P_i\cup P_i'$
for $i=1, \ldots, m$ (where we let $P_i'=J_i\cap (P_K-T)$ when $|J_i\cap (P_K-T)|=1$), and $Q_i$
for $i=1, \ldots, m-1$. Clearly,
$P$ is a  path between $u$ and $v$ and $e\in E(P)$.

It is easy to see that if $B$
is a $P$-bridge of $G$ then $B$ is a $P_K$-bridge of $K$, or a
$(P_i\cup P_i')$-bridge of $J_i$ for some $i$ with $1\le i \le m$,
or a $Q_i$-bridge of $L_i$ for some $i$ with $1\le i<m$, or $|B|=2$ and $|B\cap eCv|=|B\cap (P_K-T)|=1$.  Thus,
$P$ is a $C$-Tutte path in $G$ between $u$ and $v$ and containing $e$.
Note that
\begin{itemize}
   \item [] ${\cal J}_2=\{J_i: |J_i\cap (P_K-T)|=2 \mbox{ and }
     P_i'=J_i\cap (P_K-T)\}$ and
   \item [] ${\cal J}_3=\{J_i: |J_i\cap (P_K-T)|=2 \mbox{ and } P_i'\ne J_i\cap (P_K-T)\}$.
\end{itemize}
If we extend $P_K-T$ from $v'$ to $v$ through
$J_1, L_1, J_2, L_2, \ldots, J_{m-1}, L_{m-1}, J_m$ in order, we see that
\begin{itemize}
\item $J_1$ and $K$ double count 1 vertex (namely, $v'$);
\item $J_m$ and $K\cup (J_1\cup L_1) \cup \ldots \cup (J_{m-1}\cup L_{m-1})$ double count $c$ vertices, 
where $c=2$  when $|J_m\cap (P_K-T)|=1$  (namely, $a_m$ and $w$) and $c=4$ when $|J_m\cap (P_K-T)|=2$ (namely, $a_m, t_m$ and vertices in $V(J_m\cap P_K)$); 
 \item  $L_i$ and $K\cup (J_1\cup L_1) \cup \ldots \cup (J_{i-1}\cup
   L_{i-1})\cup J_i$ double count 1 vertex, namely  $a_i$;
 \item for $1<i<m$, if $J_i\in {\cal J}_1$ then $J_i$ and $K\cup (J_1\cup L_1) \cup
   \ldots \cup (J_{i-1}\cup L_{i-1})$ double count 2 vertices: $a_i$
   and the vertex in $V(J_i\cap P_K)$;
 \item for $1<i<m$,  if $J_i\in {\cal J}_2\cup {\cal J}_3$ then  $J_i$ and $K\cup
   (J_1\cup L_1) \cup \ldots \cup (J_{i-1}\cup L_{i-1})$ double count
   4 vertices: $a_i, t_i$, and the vertices in $V(J_i\cap P_K)$.

\end{itemize}
Note that for each $J_i \in {\cal J}_2$, the $P_K$-bridge of $K$
corresponding to the vertex $t_i\in T$  does
not contribute to  $\bridges[G]{P}$.
Thus,

\begin{align*}
\bridges[G]{P} &=\bridges[K]{P_K} + \bridges[J_1]{P_1}+\sum_{J_i \in {\cal J}_1}
                 \bridges[J_i]{P_i \cup P_i'} +\sum_{J_i\in {\cal
                 J}_2}(\bridges[J_i]{P_i\cup P_i'} - 1) +\\
                 & \quad \sum_{J_i \in {\cal
                   J}_3}\bridges[J_i]{P_i\cup
                   P_i'}+\bridges[J_m]{P_m\cup P_m'}+
                   \sum_{i=1}^{m-1}\bridges[L_i]{Q_i}.
\end{align*}

We may assume 
\begin{itemize}
\item [(10)] $eCv$ is good in $G$.
\end{itemize}
For, suppose $eCv$ is not good in $G$. Then $\taubridges[G]{ev}=2/3$. Hence, by
(4)--(9) and the above observation on double counting vertices, we have 
\begin{align*}
\bridges[G]{P} &\le \bridges[K]{P_K}+((|J_1|-1)/3-1/3) +\sum_{J_i\in
                 {\cal J}_1}(|J_i|-2)/3 +\sum_{J_i \in {\cal J}_2\cup
                 {\cal J}_3}(|J_i|-4)/3 +\\ 
             &\quad (|J_m|-c)/3+\taubridges[G]{vu}
            +\sum_{L_i\in {\cal L}}\max\{0, (|L_i|-2)/3-1/3\}\\
&\le (n-6)/3 -1/3 +\taubridges[G]{vu}+ \left(\bridges[K]{P_K}-(|K|-6)/3\right)\\
&\le  (n-6)/3 -1/3+\taubridges[G]{vu}+\taubridges[G]{ue}+1 \quad
  \mbox{ (by (3))} \\
&= (n-6)/3 +\taubridges[G]{vu}+\taubridges[G]{ue} +\taubridges[G]{ev}.
\end{align*}
So $P$ is the desired path. $\Box$

\medskip

By (10), $|L_i|\le 2$ for all $L_i\in{\cal L}$. By (4)--(10) and the above observation on double counting vertices, we have
\begin{align*}
 \bridges[G]{P} 
&\le \bridges[K]{P_K}+((|J_1|-1)/3-1/3) +\sum_{J_i\in
                 {\cal J}_1}((|J_i|-2)/3-1/3) + \\
      & \quad \sum_{J_i \in {\cal
                 J}_2}((|J_i|-4)/3+1/3-1)+ \sum_{J_i\in {\cal J}_3}((|J_i|-4)/3-1/3)
                 + (|J_m|-c)/3+\taubridges[G]{vu} \\
&\le (n-6)/3+\taubridges[G]{vu}-|{\cal J}_1|/3 -(|T|+1)/3 
  + \left(\bridges[K]{P_K}-(|K|-6)/3\right), 
\end{align*}
since $|T| = |{\cal J}_2 \cup {\cal J}_3|$. 
We may assume that 

\begin{itemize}
\item [(11)] $|T|\le 1$, ${\cal J}_i=\emptyset$ for $i=1,2,3$, and $|eCv|\ge 3$.  
\end{itemize}
First, we may assume $|T|\le 1$. For, suppose $|T|\ge 2$. Then, since   $\bridges[K]{P_K}\le
(|K|-6)/3+\taubridges[G]{ue}+1$ (by (3)),  
\[
\bridges[G]{P}\le(n-6)/3+\taubridges[G]{vu}+\taubridges[G]{ue}\le 
(n-6)/3+\taubridges[G]{vu}+\taubridges[G]{ue}+\taubridges[G]{ev},
\]
and $P$ gives the desired path. 

Therefore, $\bridges[K]{P_K}\le (|K|-6)/3+\taubridges[G]{ue}+2/3$ by (3). 
We may also assume ${\cal J}_i=\emptyset$ for $i=1,2,3$. For, otherwise, $|{\cal J}_1|\ge 1$ or $|T|\ge 1$; so  $ \bridges[G]{P} \le
(n-6)/3+\taubridges[G]{vu}-2/3+\taubridges[G]{ue}+2/3\le
(n-6)/3+\taubridges[G]{vu}+\taubridges[G]{ue}+\taubridges[G]{ev},$
and $P$ is the desired path.

If $|eCv|=1$ then  $\taubridges[G]{ev}=2/3$ and
\[
\bridges[G]{P} \le (n-6)/3+\taubridges[G]{vu}+\taubridges[G]{ue}+2/3=(n-6)/3+\taubridges[G]{vu}+\taubridges[G]{ue}+\taubridges[G]{ev};
\]
$P$ gives the desired path.   If $|eCv|=2$ then $\taubridges[G]{ev}=1/3$ and
$$ \bridges[G]{P} \le
(n-6)/3+\taubridges[G]{vu}-1/3+\taubridges[G]{ue}+2/3=
(n-6)/3+\taubridges[G]{vu}+\taubridges[G]{ue}+\taubridges[G]{ev};$$ 
so $P$ is the desired path.  Therefore, we may assume $|eCv|\ge 3$. $\Box$

\medskip

Suppose  $a_m\ne b_m$. Then by (3), (4),  (8), (10), and (11), we have 
\begin{align*} 
\bridges[G]{P} & \le \bridges[K]{P_K}+((|J_1|-1)/3-1/3)  +
((|J_m|-c)/3+\taubridges[G]{vu}-1/3)\\
   & \le  ((|K|-6)/3+\taubridges[G]{ue}+2/3)+((|J_1|-1)/3-1/3)  +
((|J_m|-c)/3+\taubridges[G]{vu}-1/3)\\
&\le (n-6)/3+\taubridges[G]{vu}+\taubridges[G]{ue}\\
&= (n-6)/3+\taubridges[G]{vu}+\taubridges[G]{ue}+\taubridges[G]{ev}.
\end{align*}
So $P$ is the desired path. 

Thus, we may  assume $a_m=b_m$. If $|J_m|=2$ then
\begin{align*}
\bridges[G]{P} & = \bridges[K]{P_K} \\
& \le (|K|-6)/3+\taubridges[G]{ue}+2/3 \quad(\mbox{by (3) and (11)})\\
& \le (n-6)/3+\taubridges[G]{ue}-1/3 \quad (\mbox{since $|eCv|\ge 3$})\\
&\le (n-6)/3+\taubridges[G]{vu}+\taubridges[G]{ue}+\taubridges[G]{ev}, 
\end{align*}
and $P$ is the desired path. 
So assume $|J_m|\ge 3$. Then $\taubridges[G]{vu}=2/3$. Since $|eCv|\ge 3$,
\begin{align*}
   \bridges[G]{P} & =\bridges[K]{P_K}+1\\
  &\le (|K|-6)/3+\taubridges[G]{ue}+2/3+1 \quad(\mbox{by (3) and (11)})\\
  &\le (n-6)/3+\taubridges[G]{ue}+2/3 \quad (\mbox{since $|eCv|\ge 3$})\\
   & \le (n-6)/3+\taubridges[G]{vu}+\taubridges[G]{ue}+\taubridges[G]{ev}.
\end{align*}
Again, $P$ gives the desired path. 
\qed

\section{Proof of Theorem~\ref{main}}

Note that we need $n\ge 6$.  However, when $n\le 5$ the graph $G$ is
Hamiltonian.

First, suppose $G$ is 4-connected. Fix a planar drawing of $G$ and let
$T$ be the outer cycle of $G$.
Let $uv, e\in E(T)$ be distinct. By
applying Theorem~\ref{technical}, $G$ has a $T$-Tutte path $P$ between
$u$ and $v$ such that $e\in E(P)$. Since $G$ is 4-connected,
$\bridges[G]{P}=0$; so $P$ is in
fact a Hamilton path. Hence,  $P+uv$ is a Hamilton cycle in $G$ and,
thus, has length $n$, which is at least $(2n+6)/3$ (as $n\ge 6$). 

Hence, we may assume that $G$ is not 4-connected. Then, since $G$ is
essentially 4-connected, there exists $x\in V(G)$ such that $x$ has
degree 3 in $G$. So let $N_G(x)=\{u,v,w\}$ and let $H:=G-x$ and assume
that $H$ is a plane graph with $u,v,w$ on the outer cycle $C$ of $H$ in
counter clockwise order. Note that $(H,C)$ is a circuit graph.

Suppose two of $\order{uCw}, \order{wCv}, \order{vCu}$ is at least
3. Without loss of generality, we may assume
that $\order{uCw}\ge 3$ and $\order{wCv}\ge 3$. Let $e\in E(uCw)$ be incident with $w$. Then
$\taubridges[H]{vu}=0$, $\taubridges[H]{ue}\le 1/3$, and
$\taubridges[H]{ev}=0$. Hence, by Theorem~\ref{technical}, $H$
has a $C$-Tutte path between $u$ and $v$ such that $e\in E(P)$ and
$\bridges[H]{P}\le (n-7)/3+1/3=(n-6)/3$. Thus, $Q:=P\cup uxv$ is a Tutte cycle
in $G$ such that $\bridges[G]{Q}\le (n-6)/3$.
Since $G$ is essentially 4-connected, every $Q$-bridge is a
$K_{1,3}$. Hence, $\order{Q}\ge n-(n-6)/3=(2n+6)/3$.

So we may assume that $\order{wCv}=\order{vCu}=2$.  Consider the
plane graph $K:=H-wv$ whose outer cycle $D$ contains $vCw$. Since $G$
is essentially 4-connected, $K$ is 2-connected; so 
$(K,D)$ is a circuit graph. We can
choose $e\in E(wDv)$ incident with $w$. Now $\taubridges[K]{vu}=0$,
$\taubridges[K]{ue}\le 1/3$, and  $\taubridges[K]{ev}\le 1/3$.

If   $\taubridges[K]{ue}=0$ or $\taubridges[K]{ev}=0$,  then by Theorem~\ref{technical}, $K$
has a $D$-Tutte path between $u$ and $v$ such that $e\in E(P)$ and
$\bridges[K]{P}\le (n-7)/3+1/3=(n-6)/3$. Thus, $Q:=P\cup uxv$ is a
cycle in $G$ with $\order{Q}\ge n-(n-6)/3=(2n+6)/3$.

So assume $\taubridges[K]{ue}=\taubridges[K]{ev}=1/3$ and, hence,
$\order{wDv}=3$ and $\order{uDw}=2$. Since $n\ge 6$ and $G$ is
essentially 4-connected, one of
$\{v,w\}$ has a  neighbor inside $D$, say $w$ by
symmetry.  Now consider the plane graph $J:=H-uw$, which is
2-connected as $G$ is essentially 4-connected. Let $F$ denote the
outer cycle of $J$, which contains $\{u,v,w\}$. Clearly, $(J,F)$ is a circuit
graph. Choose $f\in E(uFw)$ incident with $w$. Then
$\taubridges[J]{ue}\le 1/3$, $\taubridges[J]{ev}=0$, and
$\taubridges[J]{vu}=0$. Hence,  by Theorem~\ref{technical}, $J$
has an $F$-Tutte path between $u$ and $v$ such that $f\in E(P)$ and
$\bridges[J]{P}\le (n-7)/3+1/3=(n-6)/3$. Thus, $Q:=P\cup uxv$ is a
cycle in $G$ with $\order{Q}\ge n-(n-6)/3=(2n+6)/3$. \qed

\end{document}